\documentclass[11pt, a4paper]{amsart}
\usepackage[utf8]{inputenc}
\usepackage[english]{babel}

\usepackage{spectralsequences}
\usepackage[margin=2.5cm]{geometry}
\usepackage{mathrsfs}
\usepackage{enumitem}
\usepackage{pinlabel}
\usepackage{graphicx}
\usepackage{amsmath,amsfonts,amssymb,mathtools}
\usepackage{tikz,tikz-cd}
\usepackage{subcaption}
\usepackage{hyperref}
\usepackage{xcolor}
\usepackage{soul}
\usetikzlibrary{calc,patterns,arrows,shapes.arrows,intersections}
\usetikzlibrary{decorations}

\newtheorem{thmintro}{Theorem}

\newtheorem{thm}{Theorem}[section]

\newtheorem{prop}[thm]{Proposition}
\newtheorem{lem}[thm]{Lemma}
\newtheorem{cor}[thm]{Corollary}
\newtheorem{rem}[thm]{Remark}
\newtheorem{example}[thm]{Example}

\newtheorem{defn}[thm]{Definition}

\newtheorem*{question}{Question}

\DeclareMathOperator{\Nerve}{\mathrm{N}}

\DeclareMathOperator{\Tor}{\mathrm{Tor}}

\newcommand{\Uast}{\mathcal{U}^{\mathrm{ast}}}

\newcommand{\cU}{\mathcal{U}}

\newcommand{\Hh}{\mathrm{H}^h}
\newcommand{\UC}{\ddot{\mathrm{C}}}
\newcommand{\UH}{\ddot{\mathrm{H}}}
\newcommand{\hh}{\mathrm{H}}

\newcommand{\N}{\mathbb{N}}

\newcommand{\Z}{\mathbb{Z}}

\newcommand{\F}{\mathbb{F}}

\newcommand{\cH}{\mathcal{H}}

\newcommand{\cZ}{\mathcal{Z}}

\newcommand{\tG}{{\tt G}}

\newcommand{\e}{\varepsilon}
\newcommand{\dd}{\delta}

\newcommand{\dH}{{\rm DH}}
\newcommand{\BH}{\ddot{\mathrm{B}}}
\renewcommand{\DH}{{\rm DH}}

\definecolor{bluedefrance}{rgb}{0.19, 0.55, 0.91}

\definecolor{aquamarine}{rgb}{0.5, 1.0, 0.83}
\definecolor{princetonorange}{rgb}{1.0, 0.56, 0.0}
\definecolor{caribbeangreen}{rgb}{0.0, 0.8, 0.6}
\definecolor{bunired}{rgb}{0.8, 0.0, 0.0}
\definecolor{cdgreen}{rgb}{0.0, 0.42, 0.24}
\definecolor{lavender(floral)}{rgb}{0.71, 0.49, 0.86}

\title[\"Uberhomology and double homology]{Bridging between \"uberhomology and double homology}

\author{Luigi Caputi}
\author{Daniele Celoria}
\author{Carlo Collari}
\date{}

\begin{document}

\maketitle

\begin{abstract}
We establish an isomorphism between the $0$-degree \"uberhomology and the double homology of finite simplicial complexes, using a Mayer-Vietoris spectral sequence argument. We clarify the correspondence between these theories by providing  examples and some consequences; in particular, we show that \"uberhomology groups detect the standard simplex, and that the double homology's diagonal is related to the connected domination polynomial.
\end{abstract}

\section*{Introduction}

Often times in mathematics, coincidences  can reveal hidden patterns or connections, leading to new insights. The aim of this paper is precisely to explore such a coincidence, involving two independently defined homology theories.

The first of the two objects we wish to compare is related to moment-angle complexes; these are central objects in the study of toric topology -- sitting at the intersection of topology, symplectic and algebraic geometry, and combinatorics, see for instance \cite{Pan13, MR3363157}.

The \textit{double homology} of the moment-angle complex $\cZ_K$, associated to a simplicial complex~$K$, was defined in \cite{LIMONCHENKO2023109274} by introducing a differential on the ordinary (simplicial) homology of~$\cZ_K$.
The study of this homology is motivated in part by stability properties of the homology $\hh_*(\cZ_K)$ with respect to deformations of $K$. This recently sparked some interest in computing double homology groups for classes of simplicial complexes~ \cite{ruiz2024sphere}, as well as finding information about its structure and, in particular, about its rank~\cite{ruiz2024sphere, han2023moment, zhang2024rank}.\\

The second algebraic object we are going to look at is called \textit{\"uberhomology}; this is a combinatorially defined, triply graded homology of simplicial complexes, introduced in~\cite{uberhomology}. 
It was later proved in~\cite{domination}, that \"uberhomology is a particular instance of \emph{poset homology}, as given in~\cite{chandler2019posets, primo}.

Restricting the triple grading yields a doubly graded homology, called $0$-degree \"uberhomology, denoted by $\BH$.
Several properties of this latter homology are indeed paralleled in double homology; for instance, the \"uberhomology of a triangulated manifold contains the fundamental class (\textit{cf.}~\cite[Thm.~7.10]{uberhomology} and~\cite[Prop.~7.2]{LIMONCHENKO2023109274}), it is well-behaved under cones and suspensions (\textit{cf}.~\cite[Prop.s~5.7  and~5.8]{MV} and \cite[Thm.~6.3]{LIMONCHENKO2023109274}), and a graph's \"uberhomology vanishes if it contains a leaf (\textit{cf}.~\cite[Prop.~5.2]{domination} and~\cite[Thm.~6.7]{LIMONCHENKO2023109274}). 

Finally, \"uberhomology is explicitly computable --albeit only for small simplicial complexes-- using freely available Python code~\cite{githububerjulius}.\\

In this paper, we establish an isomorphism between the $0$-degree \"uberhomology $\BH_*^*$ and the double homology $\DH_{*,*}$ (denoted by $\mathrm{HH}$ in \cite{LIMONCHENKO2023109274}) of finite simplicial complexes.

\begin{thmintro}\label{thm:comparisonintro}
Let $K$ be a finite and connected simplicial complex with $m$ vertices. Then, there is an isomorphism
\[ \BH_{i}^{j}(K) \cong {\rm DH}_{i-j+1,2j}(\cZ_K)\]
for each $i\neq 0,-1$. Equivalently, 
\[ {\rm DH}_{-k,2l}(\cZ_K) \cong \BH_{l-k-1}^{l}(K)\ ,\]
for each choice of $l$ and $k$ such that $ l-k \neq 0,  1$.
\end{thmintro}

To prove Theorem~\ref{thm:comparisonintro} (see Theorem~\ref{thm:comparison}), we observe that, up to a regrading, there is an identification between the double homology of $\cZ_K$ and the second page of a Mayer-Vietoris spectral sequence (Theorem~\ref{thm:double and MVss}). 
\"Uberhomology was also shown to be the second page of a Mayer-Vietoris spectral sequence~\cite{MV}. Theorem~\ref{thm:comparisonintro} follows, since the spectral sequences at hand are associated to the same cover -- the anti-star cover.

In~\cite[Rmk.~3.11]{zbMATH07820492} it was noted that double homology is the second page of a spectral sequence.
Further, in \cite[Section~4]{CVRS}, seing double homology as a page of a spectral sequence was leveraged to prove a vanishing criterion. 
A connection between double homology and the Mayer-Vietoris spectral sequence was used in \cite{AB24} (with particular reference to Remark~3.11 and Section~6) to compute higher operations on the cohomology of $\cZ_K$. 

Finally, we observe that Theorem~\ref{thm:comparisonintro} can be effectively used in both directions. For instance, on one hand, it implies that \"uberhomology detects the standard simplex (Corollary~\ref{cor:uber detects delta}). On the other hand, we use Theorem~\ref{thm:comparisonintro} to prove that the diagonal part of double homology is related to the connected domination polynomial of graphs (Corollary~\ref{cor:condom}).\\

With an eye on \"uberhomology, we  conclude with the following question:
\begin{question}
The isomorphism provided by Theorem~\ref{thm:comparisonintro} only involves \"uberhomology's lowest weighting degree. Is there a way to interpret the full \"uberhomology groups in terms of properties of moment-angle complexes?
\end{question}

\subsection*{Acknowledgements:} 
The authors are grateful to INdAM-GNSAGA.
LC~was supported by the Starting Grant 101077154 ``Definable Algebraic Topology" from the ERC.
CC~acknowledges the MUR-PRIN project 2022NMPLT8 and  the MIUR Excellence Department Project awarded to the Department of Mathematics, University of Pisa, CUP I57G22000700001. 
During the final writing stages, CC was partially supported by INdAM's ``mensilità estero'' grant, and wishes to thank Maciej Borodzik and IMPAN for their hospitality.

\section{$0$-degree  \"uberhomology }\label{sec:uber}

We start by recalling the definition of \"uberhomology, following~\cite{uberhomology, domination}. 

Let $K$ be a finite and connected simplicial complex on $m$ vertices, say $V(K)=\{v_1,\dots,v_m\}$. The ordering of the vertices will not affect the following discussion. 

A \emph{bicolouring} $\varepsilon$ on $K$ is a map $\varepsilon\colon V(K) \to \{0,1\}$. As a visual aid, we will sometimes identify~$0$ with white and $1$ with black.
A \emph{bicoloured simplicial complex} is a pair $(K,\e)$, consisting of a simplicial complex $K$ equipped with the bicolouring~$\e$.
Given a $n$-dimensional simplex~$\sigma$ in $(K,\e)$, define its $\e$-\emph{weight} as
\begin{equation}\label{eq:weight}
w_\varepsilon(\sigma )\coloneqq 
n+1-\sum_{v_i\in V(\sigma)} \varepsilon(v_i) \ .
\end{equation}
Equivalently, $w_\varepsilon(\sigma )$ is the number of $0$-coloured vertices in $\sigma$.
For a fixed bicolouring~$\e$, the weight in Equation~\eqref{eq:weight} induces a filtration of the simplicial chain complex~$C_*(K;\Z)$ associated to~$K$. 
More explicitly, set 
\[ \mathscr{F}_j (K, \e )  \coloneqq \Z \langle\ \sigma \mid w_{\e}(\sigma) \leq j\ \rangle \subseteq C_*(K;\Z) \ . \]
The simplicial differential $\partial$ respects this filtration, and it can be decomposed as the sum of two differentials (\emph{cf}.~\cite[Lemma~2.1]{uberhomology}); one which preserves the $w_{\e}$, denoted by~$\partial_h$, and one which decreases it by one.   Call $(C(K,\e),\partial_h)$ the bigraded chain complex, whose underlying module is $C(K;\Z)$; the first degree is given by simplices' dimensions, while the weight $w_\e$ gives the second. 
The $\e$-\emph{horizontal homology} $\Hh(K,\e)$ of $(K,\e)$ is the homology of the bigraded chain complex $(C(K,\e),\partial_h)$. In other words, $\Hh(K,\e)$ is the homology of the associated graded object for the filtration $\mathscr{F}_j (K, \e )$. 

Bicolourings on $K$ can be canonically identified with elements of the Boolean poset~$B(m)$ on a set with $m$ elements (partially ordered by inclusion).
Let $\e$ and $\e'$ be two bicolourings on $K$, differing only on the vertex $v_i$; assume further that $\e(v_i) = 0$ and~$\e'(v_i) =1$.  Denote by $d_{\e, \e'} $ the weight-preserving part of the identity map $\mathrm{Id} \colon \mathrm{H}^h(K,\e) \to \mathrm{H}^h(K,\e')$. By slightly abusing the notation, $d_{\e, \e'} $ can be seen to coincide with the linear extension of the map 
\[ d_{\e, \e'} (\sigma) = \begin{cases}(-1)^{s(\e, \e')}\sigma & \text{if }w_{\e}(\sigma) = w_{\e'}(\sigma) \\ 0 &\text{otherwise}\end{cases} \ ,  \]
where $s(\e, \e') = \left\vert \{ j \in V\; : \; j<i,\; \e(v_j)=1\}\right\vert$ -- see~\cite[Section~6]{uberhomology} and~\cite[Section~2]{domination}.
Note that $d_{\e, \e'} (\sigma) = 0$ occurs if and only if $w_{\e}(\sigma) = w_{\e'}(\sigma)-1$.

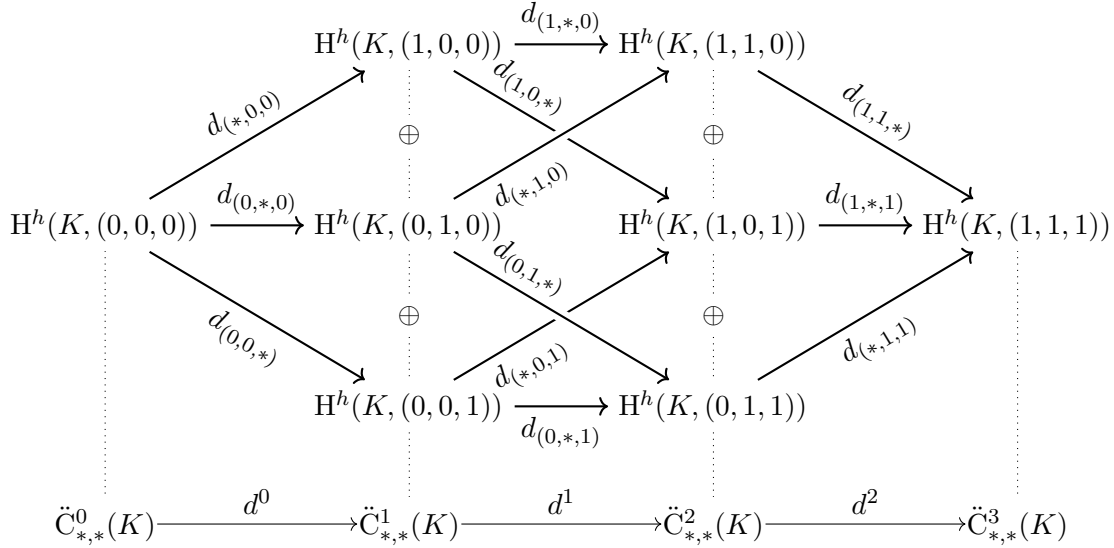
\begin{figure}[ht]
\begin{tikzpicture}[scale = .8]
\draw[dotted] (0,0) -- (0,-4.5) node[below] {$\UC_{*,*}^0(K)$};

\draw[dotted] (5,3) -- (5,-4.5) node[below] {$\UC_{*,*}^1(K)$};

\draw[dotted] (10,3) -- (10,-4.5) node[below] {$\UC_{*,*}^2(K)$};

\draw[dotted] (15,0) -- (15,-4.5) node[below] {$\UC_{*,*}^3(K)$};

\node[below] (uc0) at (0,-4.5) {\phantom{$\UC^0(K)$}};
\node[below] (uc1) at (5,-4.5) {\phantom{$\UC^0(K)$}};
\node[below] (uc2) at (10,-4.5) {\phantom{$\UC^0(K)$}};
\node[below] (uc3) at (15,-4.5) {\phantom{$\UC^0(K)$}};

\draw[->] (uc0) -- (uc1) node[midway, above] {$d^0$};
\draw[->] (uc1) -- (uc2) node[midway, above] {$d^1$};
\draw[->] (uc2) -- (uc3) node[midway, above] {$d^2$};

\node[fill, white] at (5,1.5){$\oplus$};
\node[fill, white] at (5,-1.5){$\oplus$};
\node  at (5,1.5){$\oplus$};
\node  at (5,-1.5){$\oplus$};

\node[fill, white] at (10,1.5){$\oplus$};
\node[fill, white] at (10,-1.5){$\oplus$};
\node  at (10,1.5){$\oplus$};
\node  at (10,-1.5){$\oplus$};

\node[fill, white] at (0,0) {${\Hh(K,(0,0,0))}$};

\node[fill, white] at (5,3){${\Hh(K,(1,0,0))}$};
\node[fill, white] at (5,0){${\Hh(K,(0,1,0))}$};
\node[fill, white] at (5,-3){${\Hh(K,(0,0,1))}$};

\node[fill, white] at (10,3) {${\Hh(K,(1,1,0))}$};
\node[fill, white] at (10,0) {${\Hh(K,(1,0,1))}$};
\node[fill, white] at (10,-3) {${\Hh(K,(0,1,1))}$};

\node[fill, white] at (15,0) {${\Hh(K,(1,1,1))}$};

\node (a) at (0,0) {${\Hh(K,(0,0,0))}$};

\node (b1) at (5,3) {${\Hh(K,(1,0,0))}$};
\node (b2) at (5,0) {${\Hh(K,(0,1,0))}$};
\node (b3) at (5,-3){${\Hh(K,(0,0,1))}$};

\node (c1) at (10,3) {${\Hh(K,(1,1,0))}$};
\node (c2) at (10,0) {${\Hh(K,(1,0,1))}$};
\node (c3) at (10,-3) {${\Hh(K,(0,1,1))}$};

\node (d) at (15,0) {${\Hh(K,(1,1,1))}$};

\draw[thick, ->] (a) -- (b1) node[midway,above,rotate =31] {$d_{(*,0,0)}$}; 
\draw[thick, ->] (a) -- (b2) node[midway,above] {$d_{(0,*,0)}$}; 
\draw[thick, ->] (a) -- (b3) node[midway,below,rotate =-29] {$d_{(0,0,*)}$}; 

\draw[thick, ->] (b1) -- (c1) node[midway,above] {$d_{(1,*,0)}$}; 
\draw[thick, ->] (b1) -- (c2) node[midway,above left,rotate =-29] {$d_{(1,0,*)}$}; 

\draw[thick, ->] (b3) -- (c2) node[midway,below left,rotate =31] {$d_{(*,0,1)}$}; 
\draw[thick, ->] (b3) -- (c3) node[midway,below] {$d_{(0,*,1)}$}; 

\draw[line width = 5, white] (b2) -- (c1) ; 
\draw[line width = 5, white] (b2) -- (c3) ; 
\draw[thick, ->] (b2) -- (c1) node[midway,below left,rotate =31] {$d_{(*,1,0)}$} ; 
\draw[thick, ->] (b2) -- (c3) node[midway,above left,rotate =-29] {$d_{(0,1,*)}$}; 

\draw[thick, <-] (d) -- (c1) node[midway,above,rotate =-29] {$d_{(1,1,*)}$}; 
\draw[thick, <-] (d) -- (c2) node[midway,above] {$d_{(1,*,1)}$}; 
\draw[thick, <-] (d) -- (c3) node[midway,below,rotate =31] {$d_{(*,1,1)}$}; 
\end{tikzpicture}

\caption{For a simplicial complex $K$ with $3$ vertices, the horizontal homologies of $K$ are placed over the vertices of a 3d cube, whose edges are components of the \"uberhomology differentials.}
\label{fig:cubo}
\end{figure}

For a given bicolouring $\e$ on $K$,  set $\ell(\e) \coloneqq \sum_{j} \e(v_j) $. The \emph{$j$-th \"uber chain module} is then defined as 
\begin{equation}\label{eq:uberchain}
\UC^{j}(K;\Z) = \bigoplus_{\ell(\e) = j} \Hh(K,\e;\Z) \ .
\end{equation}
By \cite[Proposition~6.2]{uberhomology}, the map 
\begin{equation}\label{eq:uberdiff}
\ddot{d}^j\coloneqq \sum_{\ell(\e) = j} d_{\e,\e'}\colon \UC^{j}(K;\Z)\to \UC^{j+1}(K;\Z)
\end{equation}
is a differential of degree $1$, turning $\left(\UC^{*}(K;\Z), \ddot{d}\right)$ into a cochain complex.
A schematic summary for the construction of the \"uber chain complex is presented in~Figure~\ref{fig:cubo}.

\begin{defn}\label{def:uber}
The \emph{\"uberhomology} $\UH^*(K)$ of a finite and connected simplicial complex $K$ is the homology of the complex $\left(\UC^{*}(K;\Z), \ddot{d}\right)$.
\end{defn}

\"Uberhomology groups can be endowed with two extra gradings, yielding a triply graded module. Indeed, the differential $\ddot{d}$ preserves both the simplices' weight and dimension. 
The notation for the three gradings on the \"uberhomology is as follows: $\UH^j_{k,i}(K)$ denotes the component of the homology generated by simplices of dimension $i$, with $k$ vertices of colour $0$, and whose (\"uber)homological degree is $j$.

\begin{rem}\label{uberposet}
It was observed in~\cite{domination} that the definition of \"uberhomology can be rephrased in terms of poset homology~\cite{chandler2019posets, primo}. 
With this interpretation, the choice of the signs in the definition of $d_{\e, \e'}$ is a sign assignment in the sense of~\cite[Definition~2.9]{domination}, and different sign assignments yield isomorphic homology groups -- see~\cite[Remark~2.10]{domination} and~\cite[Example~3.15]{primo}. Moreover, Definition~\ref{def:uber} was extended to general coefficients in~\cite{domination}.
\end{rem}

As mentioned above, the \"uberdifferential $\ddot{d}$ preserves the $(k,i)$-bidegree. In particular,  specialising to the component of \"uberhomology of weight $0$ yields a bigraded homology.

\begin{defn}
For a simplicial complex $K$, define the $0$-degree \"uberhomology to be the bigraded homology
\[
\BH^j_i (K) \coloneqq \UH^j_{0,i}(K) \ .
\]
\end{defn}

In view of Remark~\ref{uberposet}, we can provide an alternative description of $\BH(K)$ as poset homology: 

\begin{lem}
The bigraded homology $\BH^*_i(K)$ is the poset homology of the Boolean poset $B(m)$, with coefficients in the $i$-th simplicial homology functor.
\end{lem}

\begin{proof}
    Given a subset $I\subseteq V(K)$, denote by~$K[I]\subseteq K$ the sub-complex induced by $I$. 
We include the case $I = \emptyset$, and set $K[\emptyset] = \emptyset$ to be the empty simplicial complex.
For $\varepsilon \in \Z_2^m$ define $I_\varepsilon$ as the  
set of $1$-coloured vertices with respect to $\varepsilon$. Observe that the weight 0 part of the  horizontal homology $\Hh(K,\e;\Z)$ is the simplicial homology of $K[I_\e]$. By Equation~\eqref{eq:uberchain}, the homology $\BH^*_i(K)$ is then obtained by decorating each vertex~$\varepsilon$ in the Boolean poset~$B(m)$ with the $i$-\textit{th} homology~$ \mathrm{H}_i (K[I_\varepsilon])$ of $K[I_\varepsilon]$; 
the differentials associated to the cube's edges are induced by inclusion, as in Equation~\eqref{eq:uberdiff}. Therefore, if $m=|V(K)|$, we obtain a functor
\[
\mathcal{H}_i\colon B(m)\to \mathbf{Ab}
\]
which associates to each $\e$ the homology $ \mathrm{H}_i (K[I_\varepsilon])$.
By definition of poset homology -- see also \cite[Proposition~2.14]{domination} -- we get that the \"uberhomological construction computes the poset homology of the Boolean poset~$B(m)$, with coefficients in $\mathcal{H}_i$. This proves the statement.
\end{proof} 

\begin{example}
The chain complex for the homology $\BH^j_i(\partial \Delta^2)$ of the boundary of $\Delta^2$ is shown in Figure~\ref{fig:triang}. This complex is supported in homological degrees between $1$ and $3$, and simplicial degrees $0$ and $1$. In degree $i=0$, it is isomorphic to the simplicial chain complex associated to $\partial \Delta^{2}$, while in degree $i=1$ there are only trivial differentials, and a unique non-trivial summand in degree $j=3$. It follows that 
\[ \BH^{j}_{i}(\partial\Delta^2) = \begin{cases} \Z  & \text{if }(j,i) =  (1,0),(3,1),\\ 0 & \text{otherwise.}\end{cases}\]
More explicitly, the generator in bidegree $(1,0)$ is spanned by the direct sum of the three connected components identified by a single black vertex (see the first column of Figure~\ref{fig:triang}). The other generator coincides with the fundamental class of $\partial \Delta^2$, seen as a triangulation of $S^1$. 

\begin{figure}[h!]
\begin{tikzpicture}[scale = .75]
\draw[dotted] (0,0) -- (0,-5.5) node[below] {\phantom{$\UC^0 (K)$}};

\draw[dotted] (5,3) -- (5,-5.5) node[below] {\phantom{$\UC^0 (K)$}};

\draw[dotted] (10,3) -- (10,-5.5) node[below] {\phantom{$\UC^0 (K)$}};

\draw[dotted] (15,0) -- (15,-5.5) node[below] {\phantom{$\UC^0 (K)$}};

\node[below] (uc0) at (0,-5.5) {{$_{\phantom{(0)}}0^{\phantom{(0)}}$}};
\node[below] (uc1) at (5,-5.5) {{$\Z_{(0)}^3$}};
\node[below] (uc2) at (10,-5.5) {{$\Z_{(0)}^3$}};
\node[below] (uc3) at (15,-5.5) {{$\Z_{(0)}^{\phantom{(0)}}\oplus \Z^{\phantom{(0)}}_{(1)}$}};

\draw[->] (uc0) -- (uc1) node[midway, above] {};
\draw[->] (uc1) -- (uc2) node[midway, above] {};
\draw[->] (uc2) -- (uc3) node[midway, above] {};

\node[fill, white] at (5,1.5){$\oplus$};
\node[fill, white] at (5,-1.5){$\oplus$};
\node  at (5,1.5){$\oplus$};
\node  at (5,-1.5){$\oplus$};

\node[fill, white] at (10,1.5){$\oplus$};
\node[fill, white] at (10,-1.5){$\oplus$};
\node  at (10,1.5){$\oplus$};
\node  at (10,-1.5){$\oplus$};

\node[fill, white] at (0,0) {${\rm H}_*\left(\text{\raisebox{-1em}{\begin{tikzpicture}[scale =1, very thick]
    \node (a) at (0,0) {} ;
    \node (b) at (1,0) {};
    \node (c) at (.5, .866) {};
    \draw[white] (0,0) circle (0.05) ;
    \draw[white] (1,0) circle (.05);
    \draw[white] (.5, .866) circle (.05);
    \draw[white] (a) -- (b);
    \draw[white] (c) -- (b);
    \draw[white] (a) -- (c);
    \end{tikzpicture}}}\right)$};

\node[fill, white] at (5,3){${\rm H}_*\left(\text{\raisebox{-1em}{\begin{tikzpicture}[scale =1, very thick]
    \node (a) at (0,0) {} ;
    \node (b) at (1,0) {};
    \node (c) at (.5, .866) {};
    \draw[, white] (0,0) circle (0.05) ;
    \draw[, white] (1,0) circle (.05);
    \draw[, white] (.5, .866) circle (.05);
    \draw[white] (a) -- (b);
    \draw[white] (c) -- (b);
    \draw[white] (a) -- (c);
    \end{tikzpicture}}}\right)$};
\node[fill, white] at (5,0){${\rm H}_*\left(\text{\raisebox{-1em}{\begin{tikzpicture}[scale =1, very thick]
    \node (a) at (0,0) {} ;
    \node (b) at (1,0) {};
    \node (c) at (.5, .866) {};
    \draw[, white] (0,0) circle (0.05) ;
    \draw[, white] (1,0) circle (.05);
    \draw[, white] (.5, .866) circle (.05);
    \draw[white] (a) -- (b);
    \draw[white] (c) -- (b);
    \draw[white] (a) -- (c);
    \end{tikzpicture}}}\right)$};
\node[fill, white] at (5,-3){${\rm H}_*\left(\text{\raisebox{-1em}{\begin{tikzpicture}[scale =1, very thick]
    \node (a) at (0,0) {} ;
    \node (b) at (1,0) {};
    \node (c) at (.5, .866) {};
    \draw[, white] (0,0) circle (0.05) ;
    \draw[, white] (1,0) circle (.05);
    \draw[, white] (.5, .866) circle (.05);
    \draw[white] (a) -- (b);
    \draw[white] (c) -- (b);
    \draw[white] (a) -- (c);
    \end{tikzpicture}}}\right)$};

\node[fill, white] at (10,3) {${\rm H}_*\left(\text{\raisebox{-1em}{\begin{tikzpicture}[scale =1, very thick]
    \node (a) at (0,0) {} ;
    \node (b) at (1,0) {};
    \node (c) at (.5, .866) {};
    \draw[, white] (0,0) circle (0.05) ;
    \draw[, white] (1,0) circle (.05);
    \draw[, white] (.5, .866) circle (.05);
    \draw[white] (a) -- (b);
    \draw[white] (c) -- (b);
    \draw[white] (a) -- (c);
    \end{tikzpicture}}}\right)$};
\node[fill, white] at (10,0) {${\rm H}_*\left(\text{\raisebox{-1em}{\begin{tikzpicture}[scale =1, very thick]
    \node (a) at (0,0) {} ;
    \node (b) at (1,0) {};
    \node (c) at (.5, .866) {};
    \draw[, white] (0,0) circle (0.05) ;
    \draw[, white] (1,0) circle (.05);
    \draw[, white] (.5, .866) circle (.05);
    \draw[white] (a) -- (b);
    \draw[white] (c) -- (b);
    \draw[white] (a) -- (c);
    \end{tikzpicture}}}\right)$};
\node[fill, white] at (10,-3) {${\rm H}_*\left(\text{\raisebox{-1em}{\begin{tikzpicture}[scale =1, very thick]
    \node (a) at (0,0) {} ;
    \node (b) at (1,0) {};
    \node (c) at (.5, .866) {};
    \draw[, white] (0,0) circle (0.05) ;
    \draw[, white] (1,0) circle (.05);
    \draw[, white] (.5, .866) circle (.05);
    \draw[white] (a) -- (b);
    \draw[white] (c) -- (b);
    \draw[white] (a) -- (c);
    \end{tikzpicture}}}\right)$};

\node[fill, white] at (15,0) {${\rm H}_*\left(\text{\raisebox{-1em}{\begin{tikzpicture}[scale =1, very thick]
    \node (a) at (0,0) {} ;
    \node (b) at (1,0) {};
    \node (c) at (.5, .866) {};
    \draw[, white] (0,0) circle (0.05) ;
    \draw[, white] (1,0) circle (.05);
    \draw[, white] (.5, .866) circle (.05);
    \draw[white] (a) -- (b);
    \draw[white] (c) -- (b);
    \draw[white] (a) -- (c);
    \end{tikzpicture}}}\right)$};

\node (a0) at (0,0) {{${\rm H}_*\left(\text{\raisebox{-1em}{\begin{tikzpicture}[scale =1, very thick]
    \node (a) at (0,0) {} ;
    \node (b) at (1,0) {};
    \node (c) at (.5, .866) {};
    \draw[fill, gray, opacity = .2] (0,0) circle (0.05) ;
    \draw[fill, gray, opacity = .2] (1,0) circle (.05);
    \draw[fill, gray, opacity = .2] (.5, .866) circle (.05);
    \draw[gray, opacity = .2] (a) -- (b);
    \draw[gray, opacity = .2] (c) -- (b);
    \draw[gray, opacity = .2] (a) -- (c);
    \end{tikzpicture}}}\right)$}};

\node (b1) at (5,3) {{${\rm H}_*\left(\text{\raisebox{-1em}{\begin{tikzpicture}[scale =1, very thick]
    \node (a) at (0,0) {} ;
    \node (b) at (1,0) {};
    \node (c) at (.5, .866) {};
    \draw[fill, black] (0,0) circle (0.05) ;
    \draw[fill, gray, opacity = .2] (1,0) circle (.05);
    \draw[fill, gray, opacity = .2] (.5, .866) circle (.05);
    \draw[gray, opacity = .2] (a) -- (b);
    \draw[gray, opacity = .2] (c) -- (b);
    \draw[gray, opacity = .2] (a) -- (c);
    \end{tikzpicture}}}\right)$}};
\node (b2) at (5,0) {{${\rm H}_*\left(\text{\raisebox{-1em}{\begin{tikzpicture}[scale =1, very thick]
    \node (a) at (0,0) {} ;
    \node (b) at (1,0) {};
    \node (c) at (.5, .866) {};
    \draw[fill, gray, opacity = .2] (0,0) circle (0.05) ;
    \draw[fill, black] (1,0) circle (.05);
    \draw[fill, gray, opacity = .2] (.5, .866) circle (.05);
    \draw[gray, opacity = .2] (a) -- (b);
    \draw[gray, opacity = .2] (c) -- (b);
    \draw[gray, opacity = .2] (a) -- (c);
    \end{tikzpicture}}}\right)$}};
\node (b3) at (5,-3){{${\rm H}_*\left(\text{\raisebox{-1em}{\begin{tikzpicture}[scale =1, very thick]
    \node (a) at (0,0) {} ;
    \node (b) at (1,0) {};
    \node (c) at (.5, .866) {};
    \draw[fill, gray, opacity = .2] (0,0) circle (0.05) ;
    \draw[fill, gray, opacity = .2] (1,0) circle (.05);
    \draw[fill, black] (.5, .866) circle (.05);
    \draw[gray, opacity = .2] (a) -- (b);
    \draw[gray, opacity = .2] (c) -- (b);
    \draw[gray, opacity = .2] (a) -- (c);
    \end{tikzpicture}}}\right)$}};

\node (c1) at (10,3) {{${\rm H}_*\left(\text{\raisebox{-1em}{\begin{tikzpicture}[scale =1, very thick]
    \node (a) at (0,0) {} ;
    \node (b) at (1,0) {};
    \node (c) at (.5, .866) {};
    \draw[fill, black] (0,0) circle (0.05) ;
    \draw[fill, black] (1,0) circle (.05);
    \draw[fill, gray, opacity = .2] (.5, .866) circle (.05);
    \draw[black] (a) -- (b);
    \draw[gray, opacity = .2] (c) -- (b);
    \draw[gray, opacity = .2] (a) -- (c);
    \end{tikzpicture}}}\right)$}};
\node (c2) at (10,0) {{${\rm H}_*\left(\text{\raisebox{-1em}{\begin{tikzpicture}[scale =1, very thick]
    \node (a) at (0,0) {} ;
    \node (b) at (1,0) {};
    \node (c) at (.5, .866) {};
    \draw[fill, black] (0,0) circle (0.05) ;
    \draw[fill, gray, opacity = .2] (1,0) circle (.05);
    \draw[fill, black] (.5, .866) circle (.05);
    \draw[gray, opacity = .2] (a) -- (b);
    \draw[gray, opacity = .2] (c) -- (b);
    \draw[black] (a) -- (c);
    \end{tikzpicture}}}\right)$}};
\node (c3) at (10,-3){{${\rm H}_*\left(\text{\raisebox{-1em}{\begin{tikzpicture}[scale =1, very thick]
    \node (a) at (0,0) {} ;
    \node (b) at (1,0) {};
    \node (c) at (.5, .866) {};
    \draw[fill, gray, opacity = .2] (0,0) circle (0.05) ;
    \draw[fill, black] (1,0) circle (.05);
    \draw[fill, black] (.5, .866) circle (.05);
    \draw[gray, opacity = .2] (a) -- (b);
    \draw[black] (c) -- (b);
    \draw[gray, opacity = .2] (a) -- (c);
    \end{tikzpicture}}}\right)$}};

\node (d) at (15,0) {{${\rm H}_*\left(\text{\raisebox{-1em}{\begin{tikzpicture}[scale =1, very thick]
    \node (a) at (0,0) {} ;
    \node (b) at (1,0) {};
    \node (c) at (.5, .866) {};
    \draw[fill, black] (0,0) circle (0.05) ;
    \draw[fill, black] (1,0) circle (.05);
    \draw[fill, black] (.5, .866) circle (.05);
    \draw[black] (a) -- (b);
    \draw[black] (c) -- (b);
    \draw[black] (a) -- (c);
    \end{tikzpicture}}}\right)$}};

\draw[thick, ->] (a0) -- (b1);
\draw[thick, ->] (a0) -- (b2); 
\draw[thick, ->] (a0) -- (b3);

\draw[thick, ->] (b1) -- (c1); 
\draw[thick, ->] (b1) -- (c2);

\draw[thick, ->] (b3) -- (c2);
\draw[thick, ->] (b3) -- (c3);
\draw[line width = 5, white] (b2) -- (c1) ; 
\draw[line width = 5, white] (b2) -- (c3) ; 
\draw[thick, ->] (b2) -- (c1);
\draw[thick, ->] (b2) -- (c3);

\draw[thick, <-] (d) -- (c1) ;
\draw[thick, <-] (d) -- (c2) ; 
\draw[thick, <-] (d) -- (c3) ;
\end{tikzpicture}

\caption{The $0$-degree \"uberchain complex of $\partial\Delta^2$. Here $\Z^{d}_{(i)}$ denotes a $\Z^d$ summand in $\BH^{*}_{i}$. }\label{fig:triang}
\end{figure}
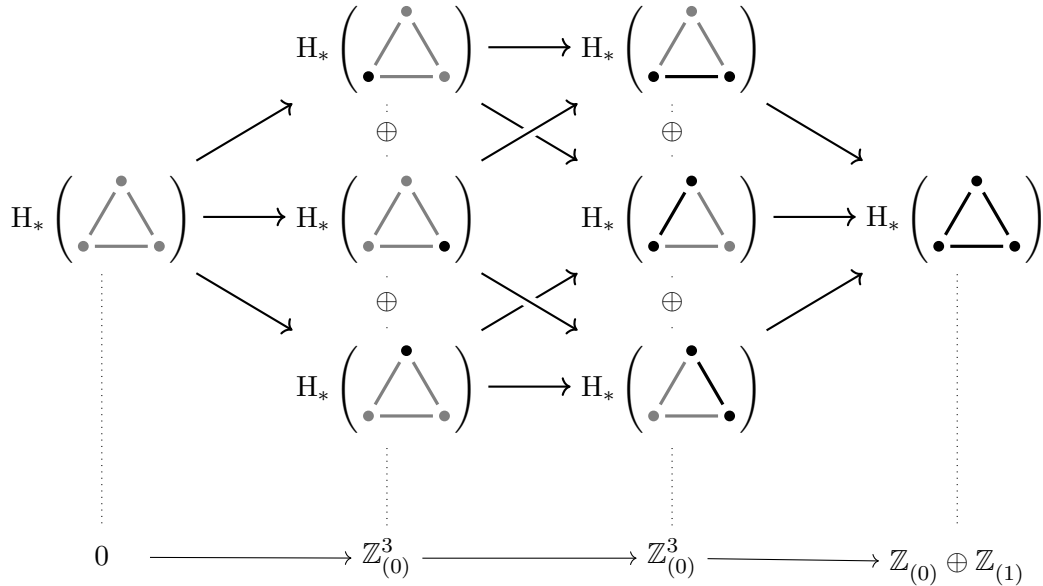
\end{example}

More generally, one can compute the \"uberhomology of the boundary of the simplex in any dimension,~\emph{cf.}~\cite[Example~4.7]{MV}, we recall the result here.

\begin{example}\label{ex:ubersimplspheres}
Let $K$ be the standard simplicial sphere $\partial\Delta^p$ of dimension $p-1$. Then, the only non-trivial $0$-degree \"uberhomology groups of $K$ are $\BH^{p+1}_{p-1}(K) =\Z$ and $\BH^1_0(X) = \Z$.
\end{example}

We refer to \cite{uberhomology,domination,MV} for further computations of \"uberhomology groups. We conclude this section by observing that in~\cite{domination}, the authors proved the existence of a relation between certain \"uberhomology groups and the connected domination polynomial~{(\emph{cf.}~Corollary~\ref{cor:condom}).}

\section{Double homology}

In this section we recall the definition of double homology, following~\cite{LIMONCHENKO2023109274}.

Let $K$ be a simplicial complex. 
For each $I\subseteq V=V(K)$,  consider the reduced simplicial homology $\widetilde{\hh}_*(K[I];\Z)$ with coefficients in  $\Z$ (when clear from the context, coefficients will be omitted). We adopt the convention that the group $\widetilde{\hh}_p(K[\emptyset])=\Z$ is non-trivial only for $p = -1$.
For each $j\in V(K)\setminus I$,  the inclusion map~$K[I]\to K[I\cup\{j\}]$ induces a morphism $\phi_{I,j}$ between the associated homology groups.
We denote by $\phi_{p,I,j}$ the morphism induced on the homology groups in degree $p$. For each~$p\geq -1$ there exists a differential  (\textit{cf}.~\cite[Lemma 3.1]{LIMONCHENKO2023109274})
\begin{equation}\label{eq:diffdelta'}
\partial'_p\coloneqq (-1)^{p+1} \sum_{j\in V\setminus I} (-1)^{|\{i\in I \; : \; i<j\}|} \phi_{p,I,j} \ ,
\end{equation}
defined on the sum $\bigoplus_{I\subseteq V}\widetilde{\hh}_p(K[I];\Z)$. 

\begin{rem}
This construction associates to a subset of the vertices of $K$ the homology of the induced simplicial complex. This can be equivalently described as a functor on the Boolean poset~$B(m)$ on $m\coloneqq |V|$ vertices, where the poset $B(m)$ is regarded as being a category. Each object $i\in B(m)$  corresponds to a set $I_i\subseteq V$. Then, it is straightforward that $F\colon B(m)\to \mathbf{Ab}$, defined by setting $F(i)\coloneqq \widetilde{\hh}_p(K[I_i];\Z)$, yields a functor, and thus we can consider the homology of $B(m)$ with coefficients in $F$. The differential $\partial'$  from Equation~\eqref{eq:diffdelta'} is well-defined after the choice a sign assignment (in fact, any) on the cube $B(m)$ (\textit{cf}.~\cite{primo}).   
\end{rem}

To a finite simplicial complex $K$, one can also associate the well-known \textit{moment-angle complex}~$\cZ_K$~\cite[Chapter~4]{MR3363157}. The (co)homology of $\cZ_K$ has been extensively investigated (see, for instance,~\cite[Theorem~2.1]{LIMONCHENKO2023109274} and references therein).

\begin{prop}
Let $K$ be a simplicial complex with vertices $v_1, \dots, v_m$ and let $I_K$ be the ideal of $\Z[v_1,\dots,v_m]$ generated by monomials $\prod_{i\in I}v_i$ for which $I\subseteq \{1,\dots, m\}$ is not a simplex of~$K$. Then, there are isomorphisms of bigraded algebras
    \[
    \hh^*(\cZ_K)\cong \Tor_{\Z[v_1,\dots,v_m]}(\Z[K],\Z)\cong \bigoplus_{I\subseteq \{1,\dots, m\}}\widetilde{\hh}^*(K[I];\Z) \ ,
    \]
    where $\Z[K]\coloneqq \Z[v_1,\dots,v_m]/I_K$.
\end{prop}

We are interested in a certain homology theory associated to $\cZ_K$, called \emph{double homology}. Before recalling its definition, observe that there is a bigraded decomposition of the homology of $\cZ_K$, given as follows:
\begin{equation}
\label{eq:splitting}
    \hh_p(\cZ_K)\cong \bigoplus_{2l-k=p}\left(\bigoplus_{|I|=l} \widetilde{\hh}_{l-k-1} (K[I])\right) \ .
\end{equation}
The summand between parenthesis in Equation~\eqref{eq:splitting} is denoted by $ \hh_{-k,2l}(\cZ_K)$, see \cite{LIMONCHENKO2023109274}. 
With respect to this bigrading, the differential $\partial'$  of Equation~\eqref{eq:diffdelta'} splits as:
\[
\partial'_{-k,2l}\colon \hh_{-k,2l}(\cZ_K) \longrightarrow \hh_{-k-1,2l+2}(\cZ_K) \ .
\]
\begin{defn}
The \emph{double homology} of the moment-angle complex is defined as the homology
\[\dH_{*,*}(\cZ_{K}) \coloneqq \hh(\hh_{*,*}(\cZ_K),\partial')\]     
of the complex $(\hh_{*,*}(\cZ_K),\partial')$.
\end{defn}

In~\cite{LIMONCHENKO2023109274}, double homology groups are  denoted as $\hh\hh_*$; here we use $\DH_*$ instead to avoid ambiguity with Hochschild homology. 

\begin{example}[cf.~{\cite[Proposition~6.2]{LIMONCHENKO2023109274}}]\label{ex:doublehomsimplsphere}
Let $K$ be the boundary of the $(m-1)$-simplex~$\Delta^{m-1}$. Then,
    \[
    \dH_{-k,2l}(\cZ_K;\Z)=\begin{cases}
        \Z & \text{for } (-k,2l)=(0,0), (-1,2m)\\
        0 & \text{otherwise}
    \end{cases}
    \]
As a concrete example, for $K=\partial \Delta^1$ with vertices $p$ and $q$, we can  represent  the double homology chain complexes with the diagrams:
\begin{center}
\begin{tikzcd}
  & \widetilde{\hh}_0(K[p])\arrow[dr, "\phi_{\{p\},q}"]\cong 0 & \\
\widetilde{\hh}_0(\emptyset)\cong 0\arrow[ur,"\phi_{\emptyset, p}"] \arrow[dr, "\phi_{\emptyset, q}"'] & & \widetilde{\hh}_0(\partial\Delta^1)\cong \Z\\
& \widetilde{\hh}_0(K[q])\cong 0 \arrow[ur, "\phi_{\{q\}, p}"'] & 
 \end{tikzcd}    
 \end{center}
 and
\begin{center}
 \begin{tikzcd}
  & \widetilde{\hh}_{-1}(K[p])\arrow[dr, "\phi_{\{p\},q}"]\cong 0 & \\
\widetilde{\hh}_{-1}(\emptyset)\cong \Z\arrow[ur,"\phi_{\emptyset, p}"] \arrow[dr, "\phi_{\emptyset, q}"'] & & \widetilde{\hh}_{-1}(\partial\Delta^1)\cong 0\\
& \widetilde{\hh}_0(K[q])\cong 0 \arrow[ur, "\phi_{\{q\}, p}"'] & 
 \end{tikzcd}    
 \end{center}
\end{example}

In complete analogy with double homology, one can construct double cohomology as well. These two agree when the (co)homologies of the induced subcomplexes $K[I]$ are free. 

\begin{example}[{\cite[Theorem~7.2]{LIMONCHENKO2023109274}}]\label{ex:doublecycle}
Let $C_m$ be the cycle on $m\geq 5$ vertices. Then, the double cohomology $\DH^{-k,2l}(\cZ_{C_m};\Z)$ of $\cZ_{C_m}$ is $\Z$ in bidegrees $(-k,2l)=(0,0), (-1,4), (-m+3,2(m-2)), (-m+2,2m)$, and is trivial otherwise.
\end{example}

Double (co)homology groups exhibit several interesting properties. For instance, they detect simplices.

\begin{prop}[{\cite[Proposition~6.1]{LIMONCHENKO2023109274}}]\label{prop:2.6}
The moment-angle complex $\cZ_K$ is contractible if and only if $K$ is a simplex if and only if $\DH^*(\cZ_K)=\Z$.
\end{prop}

Furthermore, the double homology of flag complexes of chordal graphs is explicitly known.

\begin{prop}[{\cite[Theorem~6.8]{LIMONCHENKO2023109274}}]
If $K$ is the flag complex of a chordal graph, then the double cohomology $\DH(\cZ_K)$ has rank 2.   
\end{prop}

\section{Comparison and applications}

In this section we will show that the 0-degree of \"uberhomology, and double homology are closely related. To prove this, we are going to use the augmented Mayer-Vietoris spectral sequence with respect to a specific cover, the \emph{anti-star}. We begin by introducing the necessary notation and then provide the comparison result in Theorem~\ref{thm:comparison}.

In the following and in the next section, unless otherwise specified, we let $K$ be a simplicial complex which is not the standard simplex. We shall also assume that $K$ has no ghost vertices, that is $K[I]=\emptyset$ if and only if $I=\emptyset$. 
For each vertex $v \in V(K)$, denote by $\mathrm{ast}_K(v)$ the subcomplex of $K$ spanned by the vertices in $V(K)\setminus\{v\}$; that is the complement of the (open) star of $v$.  Denote by $\Uast(K)=\{\mathrm{ast}_K(v)\}_{v\in V(K)}$  the anti-star cover of~$K$, and by $\Nerve(\Uast)$ its nerve. 

To a $p$-simplex $\sigma \in \Nerve(\Uast)$ spanned by the vertices  $v_{j_0},\dots,v_{j_p}$, we associate the subcomplex~$U_\sigma\coloneqq \mathrm{ast}_K(v_{j_0})\cap\dots\cap \mathrm{ast}_K(v_{j_p}) \subseteq K$.

\begin{rem}
Note that, for $v\in V(K)$ a vertex, we have $\mathrm{ast}_K(v)=K[V(K)\setminus\{v\}]$. Analogously, if $\sigma$ is spanned by the vertices $v_{j_0},\dots,v_{j_p}$, we have $U_\sigma=K[V(K)\setminus \{v_{j_0},\dots,v_{j_p}\}]$. 
\end{rem}

For each $q \in \N$  and $\sigma\in \Nerve(\Uast) $, define $$\cH_q(\sigma)\coloneqq \widetilde{\hh}_q(U_\sigma)$$ as the reduced $q$-\textit{th} homology group of the subcomplex $U_\sigma$ of $K$. If $\tau$ is a face of $\sigma$, then the inclusion $U_{\sigma}\subseteq U_{\tau}$ induces a map between the associated $q$-homology groups. 
Denote by $\delta_h$ the differential induced by the simplicial differential of $\Nerve(\Uast)$. Then, if $\Nerve_p(\Uast)$ denotes the set of $p$ simplices of $\Nerve(\Uast)$, we can consider the double chain complex
\begin{equation}\label{eq:zeropage}
C_{p,q}^0\coloneqq \bigoplus_{\sigma \in \Nerve_p(\Uast)}C_q(U_\sigma)
\end{equation}
of simplicial chains $U_\sigma$ endowed with the horizontal differential $\delta_h$ and the simplicial differential $\delta_v\colon C^0_{p,q}\to C^0_{p,q-1}$ as vertical differential. Taking the filtration of this double  chain complex  (with respect to the vertical differential first) yields the so-called Mayer-Vietoris spectral sequence of $K$ with respect to the cover $\Nerve(\Uast)$. 

Recall now that the  first page of the Mayer-Vietoris spectral sequence associated to a simplicial complex $K$ and cover $\cU$ 
is 
\begin{equation}\label{eq:E1}
E^1_{p,q}=\bigoplus_{\sigma\in \mathrm{N}_p(\cU)} \hh_q(U_{\sigma}),
\end{equation}
with differential $\dd^{(1)}\colon E^1_{p,q}\to E^1_{p-1,q}$  induced by $\dd_h$. This spectral sequence is called \emph{augmented} if the empty intersection $U_{\emptyset} = K$ is allowed -- that is, in the augmented Mayer-Vietoris spectral sequence we have ${E}^1_{-1,q}={\hh}_q(K)$. We also consider a reduced and augmented version of it; that is, we consider the spectral sequence whose first page is
\begin{equation}\label{eq:E1pq}
\overline{E}^1_{p,q}=\bigoplus_{\sigma\in \mathrm{N}_p(\cU)} \widetilde{\hh}_q(U_{\sigma})\ , \quad p\geq 0, q \geq -1\ ,
\end{equation}
with the convention that if $Y = \emptyset$ then $\widetilde{H}_{-1}(Y) = \Z$ and it is trivial otherwise.
That is, we use the reduced homology groups \emph{en lieu} of the unreduced homology when taking the vertical differential of the corresponding (augmented) double chain complex.
Furthermore, we augment it by setting $\overline{E}^1_{*,-1}=\overline{E}^1_{m-1
,-1}=\Z$ and $\overline{E}^1_{-1,q}=\widetilde{\hh}_q(K)$. If $\cU$ is the anti-star cover $\Uast$, the former corresponds to the intersection of all the anti-stars, while the latter corresponds to the empty intersections of the anti-stars. We refer to~\cite{MV} for more details on the construction of the augmented unreduced Mayer-Vietoris spectral sequence. 

\begin{rem}
The `usual' (non-augmented) Mayer-Vietoris  spectral sequence, associated to a finite open covering of $X$, converges to $H_*(X)$.  In fact, being the spectral sequence associated to a double chain complex, it converges to the homology of the
associated total complex, see e.g.~\cite[Section~3.1]{MV}. Since we augmented the sequence with the homology of $X$, the augmented (unreduced) Mayer-Vietoris spectral sequence  converges to $0$ by~\cite[Remark~4.2]{MV}. 
\end{rem}

Similarly, we have the following:

\begin{lem}\label{lem:convergence}
Let $K$ be a simplicial complex which is not a standard simplex. Then, the reduced augmented Mayer-Vietoris spectral sequence, with respect to the anti-star cover, converges to $0$.
\end{lem} 
\begin{proof}
Consider the double chain complex $C^0_{p,q}$ of Equation~\eqref{eq:zeropage}, augmented in degree $p=-1$ with the simplicial chains $C_{-1,q}^0=C_q(K)$ of $K$. Recall that in the reduced version of the augmented Mayer-Vietoris spectral sequence, we also have an additional row for $q=-1$  corresponding to the augmentations of the chain complexes $C_{p,*}^0$. As in \cite[Remark~4.2]{MV}, the rows corresponding to $q\geq 0$ of this double chain complex are exact. For $q=-1$, we have $C^0_{p,-1}=\mathbb{Z}$ for all $-1\leq p\leq m-1$. The horizontal differentials, for $q=-1$ are identity maps, hence the horizontal chain complex $C^0_{*,-1}$ is acyclic as well. The total complex of a double chain complex with acyclic rows is acyclic, hence the associated spectral sequence converges to $0$.
\end{proof}

The first step towards the proof of Theorem~\ref{thm:comparison}, is to compare double homology and the second page of the reduced and augmented Mayer-Vietoris spectral sequence.

\begin{thm}\label{thm:double and MVss}
Let $K\neq \Delta^{m-1}$ be a  simplicial complex with $m$ vertices. Then,
\[ {\rm DH}_{-k,2l}(\cZ_K) \cong \overline{E}^2_{m - l - 1, l-k-1}(K)\ ,\]
for all $k,l\in \N$.
\end{thm}
\begin{proof}
By definition (see Equation~\eqref{eq:splitting}), we have
    \[
\hh_{-k,2l}(\cZ_K)= \bigoplus_{|I|=l} \widetilde{\hh}_{l-k-1} (K[I]) \ . 
    \]
This coincides with the summand in bidegree $(m-l-1,l-k-1)$ in the first page of the augmented reduced Mayer-Vietoris spectral sequence with respect to the antistar cover. 

Moreover, the differential $\partial'_{-k,2l}$ coincides, up to a sign, with the differential on the first page of the Mayer-Vietoris spectral sequence. This sign discrepancy amounts to the choice of sign assignment on a cube, and does not affect the homology~(\textit{cf.}~\cite[Section~3.2]{primo}). 

Taking the homology with respect to this differential yields the second page of the augmented and reduced Mayer-Vietoris spectral sequence.
\end{proof}

In~\cite{MV}, the authors proved that, up to a small change in gradings, the $0$-degree \"uberhomology can be identified with the second page of the augmented (unreduced) Mayer-Vietoris spectral sequence of $K$ with respect to~$\cU^{\mathrm{ast}}$: 

\begin{thm}[{\cite[Theorem~1.1]{MV}}]\label{thm:uber=MV}
Let $K$ be a finite and connected simplicial complex on $m$ vertices. Then, for all $ i\geq 0$ and $0\leq j\leq m$, there exists an isomorphism of bigraded modules 
\[
 E^2_{m-j-1, i} \cong \BH^j_i (K)
\]
between the second page of the augmented  Mayer-Vietoris spectral sequence of  ${K}$, associated to the anti-star cover, and the $0$-\textit{th} degree \"uberhomology of $K$.  
\end{thm}
\begin{figure}[h]
    \centering
    \begin{tikzpicture}[xscale = 1.35, yscale = 1.05 ]
    
    \draw[very thin, gray, opacity =.5, step =2] (-5,-4.75) grid (5,3.25);
    \draw[very thick, gray, ->] (-5.5,-2) -- (5.5,-2) node[right] {$j$};
    
    \draw[very thick, gray, ->] (-2,-4.75) -- (-2, 4) node[left] {$i$};

    \def\x{1}
    \foreach \y in {-1,0,...,2} 
       {
       \node[fill = white] (a\x\y) at (4*\x-6,2*\y-2) {$\underset{\vert I \vert = \x}{\bigoplus}$ \raisebox{-.4em}{$\widetilde{\hh}_{\y}(K[I])$}};} 
\def\x{2}
 \foreach \y in {-1,0,...,2} 
       {
       \node[fill = white] (a\x\y) at (4*\x-6,2*\y-2) {$\underset{\vert I \vert = m-2}{\bigoplus}$ \raisebox{-.4em}{$\widetilde{\hh}_{\y}(K[I])$}};} 
   
    \foreach \y in {0,...,2} 
       {\node[fill = white]  at (-4,2*\y-2) {$ \widetilde{\hh}_{\y}(K)$};
       \node[fill = white]  at (0,2*\y-2) {$ \cdots$};
       \node[fill = white]  at (4.25,2*\y-2) {$  \widetilde{\hh}_{\y}(\emptyset) = 0$};}

       \node[fill = white]  at (-4.25,-4) {$ \widetilde{\hh}_{-1}(K)$};
       \node[fill = white]  at (0,-4) {$ \cdots$};
       \node[fill = white]  at (4.25,-4) {$  \widetilde{\hh}_{-1}(\emptyset) = \Z$};
   \draw[color = white, pattern color = blue,pattern = north east lines, opacity =.25] 
   (5.25,-4.5) -- (-5.25,-4.5) -- (-5.25,3.5) -- (-3.25, 3.5) -- (-3.25, -3) --  (3.25,-3)-- (3.25, 3.5) -- (5.25,3.5) --  (5.25,-4.5);

      \draw[color = white, pattern color = red,pattern = north west lines, opacity =.25] 
   (-5.5,-.5) rectangle  (5.5,3.75);
    \end{tikzpicture}
    \caption{First page of the augmented reduced Mayer-Vietoris spectral sequence. The differentials of degree $(-1,0)$  are implicit in the picture. The augmentation is shaded in blue, while the portion isomorphic to the usual (augmented) Mayer-Vietoris spectral sequence is shaded in red.}
    \label{fig:degree-plane}
\end{figure}
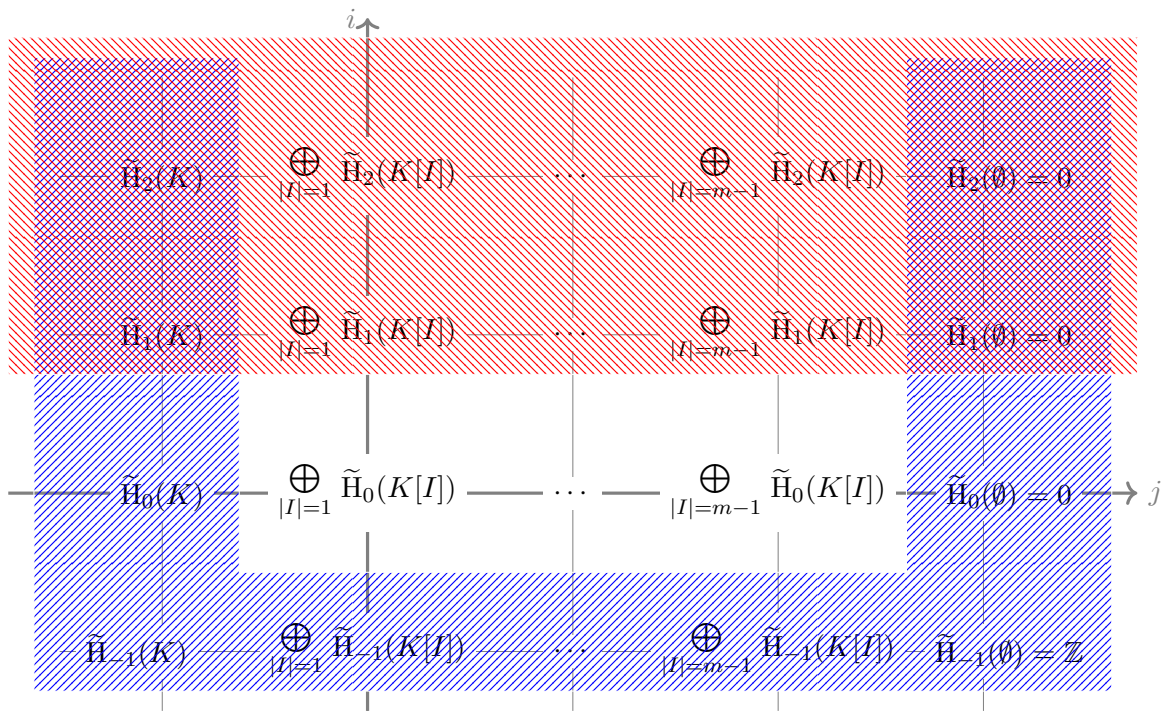

Using both the reduced and unreduced augmented Mayer-Vietoris spectral sequences associated to the anti-star cover of $K$, we obtain the main result of the paper: 
\begin{thm}\label{thm:comparison}
Let $K$ be a finite and connected simplicial complex with $m$ vertices. Then, there is an isomorphism
\[ \BH_{i}^{j}(K) \cong {\rm DH}_{i-j+1,2j}(\cZ_K)\]
for each $i\neq 0,-1$. \\Equivalently,
\[ {\rm DH}_{-k,2l}(\cZ_K) \cong \BH_{l-k-1}^{l}(K)\ ,\]
for each choice of $l$ and $k$ such that $ l-k \neq 0,  1$.
\end{thm}
\begin{proof}
If $K = \Delta^{m-1}$, since all homology groups involved vanish, the statement is true -- compare~\cite[Cor.~7.5 and Prop.~7.6]{uberhomology} and~\cite[Prop.~6.1]{LIMONCHENKO2023109274}.
If $K \neq \Delta^{m-1}$, the statement follows from Theorems~\ref{thm:double and MVss} and \ref{thm:uber=MV}, since ${\rm H}_{r}(X) \cong \widetilde{\rm H}_{r}(X)$ for~$r>0$.
\end{proof}

Indeed, the next example shows that the isomorphism between \"uberhomology groups and double homology groups can not be extended to $i=0$ (or, equivalently, $l=k+1$).

\begin{example}
Let $K$ be the boundary of $\Delta^{m-1}$, $m > 2$. From Example~\ref{ex:ubersimplspheres}, the only non-trivial terms in the $0$-degree \"uberhomology of $K$ are $\BH^{m}_{m-2}(K)$ and $\BH^1_0(K)$. 
In accordance with Theorem~\ref{thm:comparison}, we have an isomorphism $\BH^{m}_{m-2}(K) \cong \DH_{-1,2m}(\cZ_K)\cong \Z$. However, $\BH_{0}^{1}(K) \cong \Z$ and is therefore not isomorphic to $\DH_{0,2}(\cZ_K) \cong 0$ -- see Example~\ref{ex:doublehomsimplsphere}.
\end{example}

Despite this, the two homology theories are closely related even for $i=0$. 

\begin{prop}\label{prop:comparechars}
Let $K$ be a (non-empty) and connected simplicial complex with $m$ vertices and no ghost vertices. 
Then, 
\[ \chi(E^1_{*,0}(K)) =  \chi(\overline{E}^1_{*,0}(K)) + (-1)^m\ .\]
\end{prop}
\begin{proof}
It is a well-known fact that, for each non-empty simplicial complex $K'$:
\[ \hh_{0}(K') \cong \widetilde{\hh}_{0}(K') \oplus \Z\ ,\]
as $\Z$-modules, and $\hh_{0}(\emptyset) = \widetilde{\hh}_{0}(\emptyset) = 0$. Thus, 
\[ \sum_{I\subseteq V(K)} (-1)^{m-\vert I \vert -1}  {\rm rank} (\hh_{0}(K[I])) = \sum_{\tiny\begin{matrix}I\subseteq V(K)\\ I\neq \emptyset \end{matrix}} (-1)^{m-\vert I \vert -1} {\rm rank} (\hh_{0}(K[I])) =\]
\[=\sum_{\tiny\begin{matrix}I\subseteq V(K)\\ I\neq \emptyset \end{matrix}} (-1)^{m-\vert I \vert -1}  {\rm rank} (\widetilde{\hh}_{0}(K[I])) + \sum_{\tiny\begin{matrix}I\subseteq V(K)\\ I\neq \emptyset \end{matrix}} (-1)^{m-\vert I \vert -1} =\]
\[=\sum_{\tiny\begin{matrix}I\subseteq V(K)\\ I\neq \emptyset \end{matrix}} (-1)^{m-\vert I \vert -1} {\rm rank} ( \widetilde{\hh}_{0}(K[I])) + (-1)^m\ ,\]
which concludes the proof.
\end{proof}

We provide one further example featuring cycles.

\begin{example}
Let $K\coloneqq C_n$ be the cycle with $n\geq 5$ vertices. Either using Theorem~\ref{thm:uber=MV}, or by direct computation (see also~\cite[Proposition~5.6]{domination}), we get: 
\[
\BH_i^j(K)=\begin{cases}
\Z & \text{for } (i,j)=(1,n), (0,n-2)\\
0 & \text{otherwise}
\end{cases}
\]
Moreover, by Example~\ref{ex:doublecycle} (and the fact that all cohomology groups are free), we get:
\[
\dH_{-k,2l}(\cZ_K)=\begin{cases}
\Z & \text{for } (-k,2l)=(0,0), (-1,4), (-n+3,2(n-2)), (-n+2,2n)\\
0 & \text{otherwise}
\end{cases}
\]
Theorem~\ref{thm:comparison} makes $\BH_1^n(K)$ correspond with $\dH_{2-n,2n}(\cZ_K)$, and  $\BH_0^{n-2}(K)$  with $\dH_{3-n,2(n-2)}(\cZ_K)$. Again, the difference between these cohomology groups arises because the spectral sequences used to compute them are the augmented unreduced and augmented reduced Mayer-Vietoris spectral sequences, respectively.
\end{example}

\begin{example}\label{ex:icosahedron}
Consider the icosahedron $\mathcal{I}$; its double homology (with $\F_2$ coefficients) was computed in~\cite[Sec.~4.1]{ruiz2024sphere}:
$$\DH_{-k,2l}(\mathcal{Z}_\mathcal{I};\F_2) = \begin{cases}
\F_2 & \text{ if }(-k,2l) = (0,0),(-1,4),(-8,20),(-9,24)\\
\F_2^{10}& \text{ if }(-k,2l) = (-4,10),(-5,14)\\
0& \text{ otherwise }
\end{cases}$$
We can reproduce the same computation on the \"uberhomological side using J.~Frank's Python implementation~\cite{githububerjulius}:
$$\BH^j_{i}(\mathcal{I};\F_2) = \UH_{0,i}^{j}(\mathcal{I};\F_2) = \begin{cases}
\F_2 & \text{ if }(i,j) = (1,10), (2,12)\\
\F_2^{10}& \text{ if }(i,j) = (0,5),(1,7)\\
0& \text{ otherwise }
\end{cases}$$
Furthermore, all other \"uberhomology groups are trivial.
\end{example}

Recall that if $\tG$ is a connected simple graph, then a subset $S$ of vertices is called \emph{dominating} if each vertex of $\tG$ either belongs to $S$ or shares an edge with
some element of $S$. We say that $S$ is \emph{connected} if it spans a connected subgraph of $\tG$. Recall also that the connected domination polynomial is the defined as $$D_c(\tG)(t)\coloneqq \sum_S t^{|S|}\in \Z[t],$$ where $S$ ranges across all connected dominating sets of a simple and connected graph $\tG$ (see, \textit{e.g.}~\cite{domination}). 

\begin{cor}\label{cor:condom}
Let $K \neq \Delta^{m-1}$ be a connected simplicial complex with $m$ vertices; then
\[ \sum_{k} (-1)^k \mathrm{rk}( {\rm DH}_{-k,2(k+1)}(\cZ_K))= D_{c}(K^{(1)})(-1) + (-1)^{m+1}\ ,\]
where $K^{(1)}$ is the $1$-skeleton of $K$.
\end{cor}

\begin{proof}
By Theorem~\ref{thm:comparison}, we have ${\rm DH}_{-k,2(k+1)}(\cZ_K) \cong \BH_{0}^{k+1}(K)$, that is the \"uberhomology groups of $K$ in simplicial and weight degree 0. In~\cite{domination}, this restriction of \"uberhomology was called \emph{bold homology}. In particular, it was shown in~\cite[Theorem~1.2]{domination} that the Euler characteristic of bold homology computes the evaluation of the connected domination polynomial at $-1$. Now, by Proposition~\ref{prop:comparechars}, we have that the the sum 
$\sum_{k} (-1)^k \mathrm{rk}( {\rm DH}_{-k,2(k+1)}(\cZ_K))$ is equal to $\chi({E}^1_{*,0}(K)) + (-1)^{m+1}$. But $\chi({E}^1_{*,0}(K))$ is the Euler characteristic of the bold homology of~$K$. To conclude, by \cite[Lemma~2.6]{MV}, we have that the bold homology of $K$ is the bold homology of its underlying 1-skeleton $K^{(1)}$. 
\end{proof}

We note here that the augmented reduced Mayer-Vietoris spectral sequence can also be effectively used to provide alternative proof of other statements in \cite{LIMONCHENKO2023109274},  and to provide computations in double homology. As an example, \cite[Theorem~6.8]{LIMONCHENKO2023109274} can be proved as follows:
\begin{prop}
Let $K\neq \Delta^{m-1}$ be the flag complex of a chordal simple graph on $m$ vertices. Then
        \[
    \dH_{-k,2l}(\cZ_K)=\begin{cases}
        \Z & \text{for } (-k,2l)=(0,0),(-1,4)\\
        0 & \text{otherwise}
    \end{cases}
    \]
\end{prop}

\begin{proof}
Flag complexes of  chordal graphs are contractible and $1$-Leray~\cite[Lemma~3.1]{ChordalContractible}. This means that the homology of all its induced
subcomplexes is trivial for all $i\geq 1$. As a consequence, in the augmented unreduced Mayer-Vietoris spectral sequence, only the $0$-\textit{th} row  is non-trivial; further, the spectral sequence collapses at the third page. Also, the augmented reduced Mayer-Vietoris spectral sequence is non-trivial only in the  $0$-\textit{th} rows and  in position $(m-1,-1)$, corresponding to the empty set. Moreover, in such case the $(-1)$-\textit{th} column is  completely trivial, by the contractibility of $K$. 

By Proposition~\ref{prop:comparechars},  the $0$-\textit{th} row of the (second page of the)  augmented reduced Mayer-Vietoris spectral sequence  has Euler characteristic equal to $(-1)^{m+1}$. Therefore, there must be at least one non-zero class. On the other hand,  there can be at most one such class. In fact, the augmented reduced spectral sequence converges to $0$ by Lemma~\ref{lem:convergence}, and the second differential is the only differential that can kill such classes. Hence, the only non-trivial element is paired with the class corresponding to the empty set by the second differential; this  implies that there is a rank $1$ element in bidegree $(m-3,0)$.  The computation follows.
\end{proof}

Similarly, one could prove an analogous statement for $1$-Leray complexes. In such case, \cite[Theorem~1.2]{MV} implies that the spectral sequence has non-zero term in row $0$ and column $-1$. Such classes are paired up by the anti-diagonal differentials, with a further class appearing at position $(m-3,0)$.  \\

It is also possible to ``go in the opposite direction'', and use Theorem~\ref{thm:comparison} together with (the second part of) Proposition~\ref{prop:2.6} to prove a corresponding \"uberhomological statement:

\begin{cor}\label{cor:uber detects delta}
Let $K$ be a finite simplicial complex with $m$ vertices. Then, $\BH^{*}_{*}(K) \cong \BH^{1}_{0}(K) \neq 0$ if and only if $K \cong \Delta^{m-1}$.  
\end{cor}
\begin{proof}
If $K$ is a simplex, then $\BH^{*}_{*}(K) \cong \BH^{1}_{0}(K) \cong \Z$ -- see~\cite[Cor.~7.5 and Prop.~7.6]{uberhomology}.

Assume that $\BH^{*}_{*}(K) \cong \BH^{1}_{0}(K) \neq 0$.
By~Theorem~\ref{thm:comparison}, we have that 
\[ {\rm DH}_{-k,2l}(\cZ_K) \cong \BH_{l-k-1}^{l}(K) \cong 0\ ,\]
for each choice of $l$ and $k$ such that $ l-k \neq 0,  1$.
Now, from~\cite[Theorem~3.1]{domination} it follows that~$K^{(1)} \cong (\Delta^m)^{(1)}$. This allows us to show that the double homology is trivial for $l = k + 1$ as well. Indeed, we have
\[ {\rm H}_{-l+1, 2l}(\cZ_{K}) \cong \bigoplus_{\vert I \vert = l} \widetilde{\rm H}_{0}(K[I]) \cong \bigoplus_{\vert I \vert = l} \widetilde{\rm H}_{0}(K^{(1)}[I]) = 0. \]
Finally, if $l=k$, then ${\rm DH}_{-l, 2l}(\cZ_{K})$ is non-trivial only if $l=0$, where it is isomorphic to ${\rm H}_{0,0}(\cZ_{K}) = \Z$.  
The statement now follows directly from Proposition~\ref{prop:2.6}
\end{proof}

\bibliographystyle{alpha}

\end{document}